\documentclass[12pt,reqno]{amsart}
\usepackage{amssymb,latexsym,amstext,amsfonts,amscd,amsmath, }
\usepackage{graphics}
\usepackage{epsfig}
\theoremstyle{plain}
\newtheorem{theorem}{Theorem}[section]

\newtheorem*{mainA}{Theorem A}
\newtheorem*{mainB}{Theorem B}
\newtheorem{lemma}[theorem]{Lemma}
\newtheorem{defi}[theorem]{Definition}
\newtheorem{proposition}[theorem]{Proposition}
\newtheorem{corollary}[theorem]{Corollary}

\numberwithin{equation}{section}

\theoremstyle{remark}
\newtheorem{rema}[theorem]{Remark}
\newtheorem{exam}[theorem]{Example}

\renewcommand{\labelenumi}{(\roman{enumi})}
\makeatletter
\def\alphenumi{%
        \def\theenumi{\alph{enumi}}%
        \def\p@enumi{\theenumi}%
        \def\labelenumi{(\@alph\c@enumi)}}
\makeatother


\newcommand{\C}{\mathbb{C}}
\newcommand{\R}{\mathbb{R}}

\newcommand{\N}{\mathbb{N}}

\begin{document}
\pagestyle{plain}

\title{  Nuij  type pencils  of hyperbolic polynomials }
\author{Krzysztof Kurdyka}
\address{
Laboratoire de Mathematiques (LAMA), Universit\'e de Savoie\\UMR 5127
CNRS\\ \newline  73-376 Le Bourget-du-Lac cedex FRANCE }
\email{ Krzysztof.Kurdyka@univ-savoie.fr}
\author{Laurentiu Paunescu}
\address{School of Mathematics and Statistics\\
University of Sydney, NSW 2006, Australia}
\email{laurent@maths.usyd.edu.au}
\thanks{The work was partially supported by  ANR project STAAVF (France). The first author thanks the University of Sydney for  
support.
}
\dedicatory{}
\date{\today}
\maketitle
\begin{abstract}
Nuij's theorem states that if a polynomial $p\in \R[z]$ is  hyperbolic (i.e. has only real roots) then $p+sp'$ is also hyperbolic for any
$s\in \R$. We study other perturbations of hyperbolic  polynomials of the form  $p_a(z,s): =p(z) +\sum_{k=1}^d a_ks^kp^{(k)}(z)$.
We give  a full characterization of those   $a= (a_1, \dots, a_d) \in \R^d$  for which  $p_a(z,s)$ is a pencil of hyperbolic polynomials.
We give also a  full characterization of those   $a= (a_1, \dots, a_d) \in \R^d$ for which the associated families   $p_a(z,s)$
admit universal determinantal representations. In fact we show that all these sequences come from special symmetric Toeplitz matrices.
\end{abstract}

\maketitle
\section{Introduction}
Hyperbolic polynomials and the problem of their determinantal representations is  a very active area of  real algebraic geometry. A nice survey of Vinnikov \cite{V} is a  good source on recent developments in this subject. 
The  goal of this paper is a study of  $1$-parameter   families of hyperbolic polynomials and their universal determinantal representations.
  Recall that a polynomial $p\in \R[z]$ is called  {\it  hyperbolic} if all its roots are  real. Clearly  any monic hyperbolic  polynomial  of degree $d$ is  a  characteristic  polynomial of  a symmetric  $d\times d$ matrix.
  First we recall the following theorem  proved by W.
  Nuij \cite{nuij}.

 \begin{theorem}\label{NuijTheorem} Let $p\in \R[z]$ be a hyperbolic polynomial, then 
$$
p+sp^\prime
$$
is hyperbolic  for any $s\in \R$.
\end{theorem}

We give below a proof of this result, based on the  existence of  determinantal  representation of the family 
of the polynomials $p+sp^\prime, \ s \in \R$. In fact we  state and prove a generalization of Nuij's result. To this end we propose 
the following definition.

\begin{defi}\label{NuijseqDef}  
We say that  $ a= (a_1, \dots, a_d) \in \R^d$ is a  Nuij  sequence if for
any  hyperbolic  polynomial $p$ of degree $d$,   the polynomial
\begin{equation}\label{nuijseqdefeq}
p_a(z,s):  = p(z) +\sum_{k=1}^d a_ks^kp^{(k)}(z) \in \R[z],
\end{equation}
is hyperbolic for any $s\in \R$.  We denote by ${\mathcal N}_d$ the set of all Nuij sequences in $\R^d$.
\end{defi}
 Note that by Theorem \ref{NuijTheorem},   $a=(1,0, \dots , 0)$ is   a Nuij sequence for any $d\in  \N$, $d\ge 1$.
On the other hand, repeated application of Theorem \ref{NuijTheorem},  also produces Nuij sequences;  for instance we have 

$$ p +sp^\prime + s(p +sp^\prime)^\prime = p +2sp^\prime +s^2p^{\prime \prime}.
$$
Hence  $(2,1, 0, \dots, 0)$ is a Nuij sequence for any $d\in  \N$, $d\ge 2$. In Section \ref{detrep}  we shall see however that there is an
 essential difference between those two families, with respect to their  determinantal representations.

Surprisingly the set  ${\mathcal N}_d$ has a nice explicit description.
 
 \begin{mainA}\label{NuijsequenceTheorem} 
  A sequence $ a= (a_1, \dots, a_d) \in \R^d$ is a  Nuij  sequence if and only if   the polynomial
\begin{equation}\label{nuijseqdefeqQ}
q_a(z):  = z^d +\sum_{k=1}^d a_k(z^d)^{(k)} =  z^d +\sum_{k=1}^d a_k \frac {d!}{(d-k)!} z^{d-k}
\end{equation}
is hyperbolic.
 
 \end{mainA}
 
 In other words, the theorem  states that to check that a given $ a= (a_1, \dots, a_d) \in \R^d$ is  a Nuij sequence it is enough to check 
 hyperbolicity  of $p_a (z,s)$ only for $p(z)= z^d$. The proof is given in Section \ref{hyp1}; it is  based on a deep result of Borcea and Br\"and\'en
 \cite{BB} which gives a characterization of linear maps (on the  space of polynomials)  preserving hyperbolic polynomials.
 
 The second part, developed in Section \ref{detrep}, concerns universal determinantal  representation of some Nuij sequences. 
 \begin{defi}\label{UNuijseqDef} \rm 
We say that  $ a= (a_1, \dots, a_d) \in {\mathcal N}_d\subset \R^d$
 admits a {\it universal   determinantal representation}
if  there exists   a symmetric matrix  $A_a$ such that for
any  hyperbolic  polynomial $p$ of degree $d$ we have
\begin{equation}\label{nuijseqdefeqD}
p_a(z,s) =\det (zI+D+sA_a),
\end{equation}
where $D$ is  a diagonal matrix whose  characteristic   polynomial  is equal  to $p=p_a(z,0)$. 
The matrix $A_a$ will be referred    as a {\it matrix  associated to  the sequence} $ a= (a_1, \dots, a_d)$.
We denote by $\mathcal {UN}_d$ the set of all Nuij sequences in $\R^d$ which admit   universal   determinantal representations.
\end{defi}

Recall that  a square matrix is  {\it Toeplitz } if all  parallels to the principal diagonal are constant.
We say that a  symmetric Toeplitz matrix is {\it  special} if all entries outside the principal diagonal  are equal to some $\beta \in \R$, and of course
all entries on the principal diagonal are equal to some $\alpha \in\R$. 
In the sequel we will denote such a $d\times d$ matrix by $T_{\alpha,\beta}(d)$, and its determinant by  $t_{\alpha,\beta}(d):=\det T_{\alpha,\beta}(d)= (\alpha -\beta)^{d-1} (\alpha+(d-1)\beta)$. 
We obtain the following characterization of  all Nuij sequences
which admit  universal   determinantal representations.

\begin{mainB}\label{Nuijseqthm)}
A sequence 
 $ a= (a_1, \dots, a_d) \in \R^d$ is  a 
 Nuij sequence with a  universal   determinantal representation if and only if 
 there exit $\alpha,\beta \in \R$ such that
 \begin{equation}\label{Nuijseqeq}
a_i = \frac{1}{i!}t_{\alpha,\beta}(i), \, i=1, \dots, d.
\end{equation}
\end{mainB}

%


\section{Hyperbolic polynomials and Nuij sequences }\label{hyp1}

First we recall some facts about the space $\mathcal H_1^d$  of  hyperbolic (monic) polynomials of  some fixed degree $d$.
For  $x =(x_1, \dots, x_d)\in \R^d$  we have the $k$-th elementary symmetric polynomial
\begin{equation}\label{cwzor3}
c_k(x) = \sum_{i_1<\cdots < i_k}  x_{i_1}\cdots x_{i_k},
\end{equation}
for $k=1,\dots, d$.
We will identify any  $b =(b_1, \dots, b_d) \in \R^d$ with a  monic polynomial $h_b: = z^d +\sum_{k=1}^d b_kz^{d-k}$.
Thus  we can write $\mathcal H_1^d =c(\R^n),$ where  $c =(c_1, \dots, c_d):\R^d \to \R^d$ is the Vi\`ete map, hence by  the Tarski-Seidenberg theorem
it follows that $\mathcal H_1^d$ is semialgebraic.  
Moreover, the Vi\`ete map $c =(c_1, \dots, c_d):\R^d \to \R^d$ is generically a submersion, hence $\mathcal H_1^d =c(\R^n)$
has  nonempty 
interior.  In fact  $\mathcal H_1^d$ is a basic semialgebraic set which can be described using 
generalized discriminants or Bezoutians (see a nice exposition in  \cite{procesi} or even a more detailed in
\cite{rainer}). Recent developments on hyperbolic univariate polynomials are  given by Kostov in his survey  
\cite{kostov}.

For the proof of Theorem A  we need to recall several definitions and results from
\cite{BB}.

\begin{defi}\label{Def1} \cite[Definition 1]{BB} \rm
We say that a polynomial  $ f(z_1, \dots, z_n) \in \C[z_1, \dots, z_n]$ is  {\it stable} if 
$f(z_1, \dots, z_n)\ne 0$ for all $n$-tuples  $(z_1, \dots, z_n) \in \C^n$ with $im (z_j) >0,$ for
$j=1, \dots, n$. If in addition $f$ has real coefficients, it will be referred to as {\it real stable}.
The set of stable and real stable  polynomials in $n$ variables will be  denoted by  $\mathcal H_n (\C)$ and $\mathcal H_n (\R)$
respectively. Note that  for $n=1$  a polynomial $f$ is real stable, precisely means  that $f$ is hyperbolic.
\end{defi}

Let $T:\C_d[z] \to \C_d[z]$ be a linear map, where $\C_d[z] $ stands for the vector space (over $\C$) of complex polynomials of degree at most $d$. We extend it to a linear map  $T:\C_d[z,w] \to \C_d[z,w]$, by setting
$T(z^kw^l) : = T(z^k)w^l$ for all $k=1,\dots, d$ and $l\in \N$. We now state  the result which is crucial for the proof
of Theorem \ref{NuijsequenceTheorem}.
\begin{theorem}\label{BB4}\cite[Theorem 4]{BB}
Let $T:\C_d[z] \to \C_d[z]$ be a linear map. Then $T$ preserves stability if an only if either
 \begin{enumerate}
 \item  $T$ has range of dimension at most one and is of the form
 $$ T(f) = \alpha (f) P,$$
 where $\alpha : \C_d[z] \to \C$ is a linear functional and $P\in \mathcal H_1 (\C)$; or
 \item $T((z+w)^d) \in  \mathcal H_2 (\C)$.
\end{enumerate}
\end{theorem}

{\it Proof of Theorem A.}

  Assume  $ a= (a_1, \dots, a_d) \in \R^d$ is a  Nuij   sequence.  Hence by  Definition  \ref{NuijseqDef}  applied to 
  $p(z)=z^d$ with $s=1$
we obtain that the polynomial $p_a$ defined by  \eqref{nuijseqdefeq} is hyperbolic.

To prove the converse let us fix   some $ a= (a_1, \dots, a_d) \in \R^d$ and assume that
the polynomial $q_a$ defined by  \eqref{nuijseqdefeqQ} is hyperbolic. 
We associate to the sequence $ a= (a_1, \dots, a_d)$ a linear operator $T_a:\C_d[z] \to \C_d[z]$ defined by
\begin{equation}\label{leq}
T_a (p)(z):  = p(z) +\sum_{k=1}^d a_kp^{(k)}(z) \in \R[z]
\end{equation}

\begin{lemma}\label{mainlemma}
$
T_a((z+w)^d) = q_a(z+w).
$
\end{lemma}
\begin{proof} We expand first the right-hand side of  \eqref{leq}
$$
T_a((z+w)^d) =T\left(\sum_{i=0}^d  \binom{d}{i}z^i w^{d-i}\right)
= \\ \sum_{i=0}^d  \binom{d}{i}w^{d-i}T(z^i).
$$
Note that 
$$
T_a(z^i) =  \sum_{j=0}^i a_j(z^i)^{(j)}= \sum_{j=0}^i a_j(z^i)^{(j)}=
\sum_{j=0}^i a_j \frac{i!}{(i-j)!}z^{i-j}
$$
So,

$$
\sum_{i=0}^d  \binom{d}{i}w^{d-i}T(z^i) \newline =
\sum_{i=0}^d  \binom{d}{i}w^{d-i}\left( \sum_{j=0}^i a_j \frac{i!}{(i-j)!}z^{i-j} \right),
$$
hence
\begin{equation}\label{leq1}
 T_a((z+w)^d)=
\sum_{i=0}^d \frac{d!}{(d-i)!i!}z^{i-j} w^{d-i}\left( \sum_{j=0}^i a_j \frac{i!}{(i-j)!}z^{i-j} \right).
\end{equation}

On the other hand 
\begin{equation}\label{leq2}
q_a(z+w) = \sum_{i=0}^d \frac{d!}{(d-i)!} a_i(z+w)^{d-i}.
\end{equation}

$\bullet$ The coefficient in \eqref{leq1} which comes with $a_j$, $j=0,1,\dots,d$ is equal to

$$\sum_{i=j}^d \frac{d!}{(d-i)!i!}\frac{i!}{(i-j)!}z^{i-j}w^{d-i} =
\sum_{i-j=k=0}^d \frac{d!}{(d-k-j)!k!}z^{k}w^{d-j-k} 
$$

$\bullet$ The coefficient in \eqref{leq2} which comes with $a_j$, $j=0,1,\dots,d$ is equal to

$$\frac{d!}{(d-j)!}(z+w)^{d-j} 
=  \frac{d!}{(d-j)!}\sum_{k=0}^d  \binom{d-j}{k}z^k w^{d-j-k}
=
\sum_{i-j=k=0}^d \frac{d!}{(d-k-j)!}z^{k}w^{d-j-k} 
$$

Hence these coefficients are equal which proves the lemma.

\end{proof}

By the assumption $q_a$ has only real roots. 
Hence  $q_a(z+w)$ is a stable polynomial  in variables
$(z,w)$. Indeed,  if  $im (z) >0$ and $im (w) > 0$ then  $im (z+w) >0$, so $q_a(z+w) \ne 0$.
By  Lemma \ref{mainlemma}    we have $T_a((z+w)^d) = q_a(z+w)$. 
Applying Theorem \ref{BB4} we conclude that  the operator $T_a$ preserves stability, hence  $T_a$ restricted to 
$\R_d[z]$ preserves hyperbolicity. Thus we have proved that
$$p_a(z,1)  = p(z) +\sum_{k=1}^d a_kp^{(k)}(z)$$
is hyperbolic whenever $p\in \R_d[z]$ is hyperbolic.
Let us take $s\in \R^*$ and denote  $a(s): = (sa_1, \dots, s^k a_k, \dots, s^d a_d)$. 
Then the polynomial
$$q_{a(s)}(z):  = z^d +\sum_{k=1}^d s^ka_k(z^d)^{(k)} =  z^d +\sum_{k=1}^d s^ka_k \frac {n!}{(n-k)!} z^{d-k}$$
is again hyperbolic since  $q_{a(s)}(z) = s^{-d}q_{a}(sz)$.
Thus, by applying the above argument  to the sequence $a(s),$ we conclude  that
$p_a(z,s):  = p(z) +\sum_{k=1}^d a_ks^kp^{(k)}(z)$
is hyperbolic for all $s\in \R$ and any $p\in \R_d[z]$  hyperbolic. This ends the proof of Theorem A.

\subsection{Iterations of Nuij's sequences.}

Let $a=(a_1,\dots , a_d )\in \R^d$ and  $b= (b_1, \dots, b_d)\in \R^d$ be two Nuij sequences, we define their  composition
$b\circ a :=  c = (c_1, \dots, c_d)$  in the following way. For any polynomial $p(z) \in\R[z]$

$$p_c(z,s) =( p_a)_b (z,s) = 
 p_a(z,s) +\sum_{k=1}^d b_ks^k \frac{\partial^k p_a}{\partial z^k} 
  = p +\sum_{k=1}^d c_ks^kp^{(k)}.
$$

Note that with  the convention  $a_0 =b_0=1$ we have 
\begin{equation}\label{cwzor1}
c_k = \sum_{i=0}^k a_ib_{k-i}.
\end{equation}

Let   $a^1, \dots, a^r \in\R^d$. We define by   induction the composition of $r$ copies of sequences:

$$I_1(a^1) =a^1,\,\, I_r(a^1, \dots, a^r) : = I_{r-1}(a^1, \dots, a^{r-1})\circ a^r.$$

Explicitly,  if $ I_r(a^1, \dots, a^r)=c =(c_1, \dots, c_d)$, then
\begin{equation}\label{cwzor2}
c_k = \sum_{i_1<\cdots < i_r, i_1+\cdots +i_r=k}  a^1_{i_1}\cdots a^r_{i_r}.
\end{equation}

Let us consider the original Nuij sequences  of the form
\begin{equation}\label{nwzor1}
a^i =(x_i,0,\dots, 0) \in \R^d,
\end{equation}
where $x_i\in \R$, $ i=1,\dots, d$. Then, $ I_d(a^1, \dots, a^d)=c =(c_1, \dots, c_d)$  is the Nuij sequence obtained by the iteration of  $a^i$ and 
\begin{equation}\label{cwzor3}
c_k = \sum_{i_1<\cdots < i_k}  x_{i_1}\cdots x_{i_k},
\end{equation}
for $k=1,\dots, d$.
Thus  $c_k= c_k(x_1, \dots,x_d)$ is in fact the $k$-th elementary symmetric polynomial of  $x_1, \dots,x_d$.
Denote by   $c =(c_1, \dots, c_d):\R^d \to \R^d$  the Vi\`ete map and recall that   $\mathcal H_1^d =c(\R^n)$.
Thus we obtain that  $
\mathcal H_1^d\subset \mathcal N _d$.
For  $d\in \N$ let us denote by $b_d: \R^d \to \R^d$ the following linear map:
$$
b_d(a_1,\dots, a_k, \dots, a_d): =  (da_1, \dots,  \frac {d!}{(d-k)!} a_k, \dots, d! a_d).
$$

Theorem A and the above discussion can be summarized as follows.
\begin{corollary} For any $d\in \N$  we have 
$$
\mathcal H_1^d\subset \mathcal N _d = b_d^{-1} (\mathcal H_1^d).
$$
\end{corollary}

\begin{exam} For $d=2$ we have
$
\mathcal H_1^2 = \{ a_1 ^2 -4a_2 \ge 0\} \subset \mathcal N _2 = \{ a_1 ^2 -2a_2 \ge 0\}.
$

\end{exam}

\section{ Universal determinantal representations }\label{detrep}


We shall consider $1$-parameter families of hyperbolic polynomials. 
A polynomial $$p(z,s)= z^d + a_1(s)z^{d-1} + \cdots+ a_d(s)$$
 will be called a { \it pencil of hyperbolic polynomials } 
if and only if,
\begin{itemize}
\item for each $s\in \R$ the polynomial $z\mapsto p(s,z)$ is hyperbolic,

\item each coefficient $a_i(s)\in \R[s]$ is of degree at most $i$.

\end{itemize}
For any $d\ge 1,$ we shall denote  by $\mathcal{PH}_d$ the space of such  pencils  of hyperbolic polynomials.

We say that a polynomial  $p(z,s)$ admits a determinantal  representation if there are  real symmetric matrices
$A_0, A_1$ such that 

$$p(z,s) =\det (zI+A_0+sA_1),
$$
and clearly  in this case   $p(z,s)$  is a  pencil of hyperbolic polynomials.

As an easy reformulation of  a remarkable theorem of  Helton and Vinnikov \cite{HV} reads

\begin{theorem}
\label{HV}
Any polynomial $p(z,s) \in {\mathcal{P}\mathcal  H}_d$ admits a determinantal  representation.

\end{theorem}

Indeed let us  set  $z= x^{-1}$ and $s= x^{-1}y$  and finally
$$
f(x,y) := x^{d} p(z,s)=x^{d} p(x^{-1},x^{-1}y).
$$
Then $f$ is a real zero polynomial in the sense of  Helton-Vinnikov, so it has a determinantal  representation
according to Theorem 2.2 in \cite{HV}. In fact, as noticed by Lewis, Parrilo and  Ramana \cite{LPR}, Theorem \ref{HV} is a positive answer to the nonhomogeneous version of the Lax conjecture \cite{lax}.
%
%
%

We want to characterize all Nuij sequences
 $ a= (a_1, \dots, a_d) \in \R^d$ 
such that for 
any $p\in \R[z]$,  hyperbolic polynomial of degree $d$, the associated  pencil of hyperbolic  polynomials
$$
p_a(z,s):  = p +\sum_{k=1}^d a_ks^kp^{(k)} \in \R[z]
,$$ admits a {\it universal   determinantal representation}; by this we mean that
 there exists   a symmetric matrix  $A_a$ such that
$$
p_a(z,s) =\det (zI+D+sA_a),
$$
where $D$ is  a diagonal matrix. 
In other words $-D$ has on the diagonal all the roots of $p$ written in an arbitrary order. 
The matrix $A_a$ will be referred    as a {\it matrix  associated to  the sequence} $ a= (a_1, \dots, a_d)$.
We denote by $\mathcal {UN}_d$ the set of all Nuij sequences in $\R^d$ which admit   universal   determinantal representations.

%
%

\subsection{Special  Toeplitz matrices }\label{toeplitzmatrices}
Recall that  a square matrix is called {\it Toeplitz matrix } if all  parallels to the principal diagonal are constant.
We say that a  symmetric Toeplitz matrix is {\it  special} if all entries outside the principal diagonal  are equal to some $\beta \in \R$, and of course
all entries on the principal diagonal are equal to some $\alpha \in\R$. We will denote such a matrix by $T_{\alpha,\beta}$.

In the  next proposition we  will show that special Toeplitz matrices give all  Nuij sequences which admit   universal   determinantal representations. 
\begin{proposition}\label{DetNuijToeplProp} Let 
 $ a= (a_1, \dots, a_d) \in\mathcal {UN}_d.$ Then, there exists a   special Toeplitz matrix   $T_{\alpha,\beta}$
 which is associated to the sequence $a$. The constant $\alpha$ is unique. For $d=2$  we have two choices $\beta$ or $-\beta$. If $d\ge 3,$ then $\beta$ is uniquely  determined.
 
%

\end{proposition}

\begin{proof} 
Let us  fix  a sequence  $ a= (a_1, \dots, a_d) \in\mathcal {UN}_d,$ and let 
 $A_a$  be a   symmetric matrix associated to $a$. It means that
  for any    hyperbolic polynomial $p\in \R[z]$  we have 
  \begin{equation}\label{NuijDef1}
p_a(z,s) =\det (zI+D+sA_a),
\end{equation}
where $D$ is a diagonal matrix with  characteristic   polynomial equal  to $p$. We will find   a special Toeplitz matrix
$T_{\alpha,\beta}$ such that
$$
p_a(z,s) =\det (zI+D+sT_{\alpha,\beta}).
$$

As a piece of convention, we recall  that  a $j\times j$ minor of $ A_a$   is {\it principal} if it is  the  determinant  of a matrix obtained  from $A_a$ by deleting  rows and columns containing $d-j$ elements from the principal diagonal.
With  the assumption of Proposition \ref{DetNuijToeplProp}  we have. 

\begin{lemma}\label{eqminorslem1}
 For any $j=1, \dots, d,$ all $j\times j$ principal minors of $ A_a$  are equal.
\end{lemma}

Let $-\lambda_1, \dots, -\lambda_d$ be the roots of $p$. Since  $p$  can be chosen arbitrarily,
  we may consider  both sides of the identity \ref{NuijDef1} as polynomials  with real coefficients in variables
$w_i:=z+\lambda_i$, $i=1, \dots, d$.   Since $\R$ is a field of characteristic $0$,  the coefficients   corresponding to the  monomials in $w_{i_1} \cdots w_{i_j}$, where  $i_1 < \dots < i_j$,
on right-hand and left-hand side are equal. It is enough to expand both sides to check the statement of the lemma.
In particular the $1\times1$ minors, which are actually the entries on the principal  diagonal,  are all equal to some 
$\alpha \in \R$.

\begin{lemma}\label{eqminorslem2} Let  $A_a =(a_{ij})$. Then there exists $\beta \in\R$ such that 
for any distinct $i,j$ we have    $a_{ij}^2 = \beta^2$. 
 
\end{lemma}

Indeed  to each entry $a_{ij}$, $i\not =  j$  we can associate the  $2\times 2$  principal  minor
$$\det \begin{pmatrix}
&\alpha & \; &a_{ij}& \cr
&a_{ij} &\; &\alpha & 
\end{pmatrix}  =\alpha^2 -a_{ij}^2$$

Hence by Lemma \ref{eqminorslem1}  all $a_{ij}^2$ are equal for $i\not =  j$. 
We put $\beta ^2=  a_{ij}^2$. Clearly the statement of Proposition \ref{DetNuijToeplProp} is trivial for  $\beta=0$, so in the sequel we assume that $\beta\ne0$.

Before analyzing the case of $j\times j$ principal  minors, where $j\ge 3$, we need an explicit formula for the determinant of 
a special Toeplitz matrix
$T_{\alpha,\beta}$.

\begin{lemma}\label{eqminorslem3} If   $T_{\alpha,\beta}$ is a special Toeplitz matrix of size $d\times d$, then
\begin{equation}\label{sptoedet}
 t_{\alpha,\beta}(d):=\det T_{\alpha,\beta}= (\alpha -\beta)^{d-1} (\alpha+(d-1)\beta)
. \end{equation}
 \end{lemma}
 
 Next we  consider the $3\times3$ principal minors of the matrix $A_a$. We know by Lemma \ref{eqminorslem2},
  that 
for any $i\ne j$ we have $a_{ij} = \epsilon_{ij}|\beta|$, where $\epsilon_{ij}\in\{-1,1\}$. We will show that the sign of
 $\epsilon_{ij}$ can be uniformly chosen, which  means that either $\epsilon_{ij}= 1$ for all $i\ne j$,  or 
 $\epsilon_{ij}= -1$ for all $i\ne j$.
   Le us write this minor in the form 
   
  \begin{equation}\label{ND2}
\det \begin{pmatrix}
&\alpha & \; &\epsilon_{ij}|\beta|& \; &\epsilon_{ik}|\beta|& \cr
 &\epsilon_{ij}|\beta|& \; &\alpha & \;  &\epsilon_{jk}|\beta|& \cr
&\epsilon_{ik}|\beta|& \;&\epsilon_{jk}|\beta|& \; &\alpha & 
\end{pmatrix}  = \alpha^3+ 2 \epsilon_{ij}\epsilon_{ik}\epsilon_{jk}\beta^2|\beta| - 3\alpha\beta^2.
\end{equation}

By Lemma \ref{eqminorslem1} all these minors are equal, so there exists $\xi \in \{-1,1\}$ such that for all 
choices $1\le i <j<k\le d$ we have
\begin{equation}\label{signeq}
\epsilon_{ij}\epsilon_{ik}\epsilon_{jk} =\xi.
\end{equation}
This shows that   we may chose  $\epsilon_{ij}= \xi$ for all $i\ne j$.

Assume now that $d\ge 4$.
We have to show that  if  we  put  $\epsilon_{ij}=\xi$ for any $i\neq j$,  then  actually  all principal minors
$j\times j$, $j\ge 4$ are equal to the value of a   principal minor
$j\times j$, $j\ge 4$ for the original matrix $A_a$, so in fact they are determined just  by $\xi$. Note that it is enough
to  consider the case $\alpha =0$ and $\beta = 1$. First we consider the case $d=4$,  so 

$$A_a =\begin{pmatrix}
&0 & \; &\epsilon_{12}& \; &\epsilon_{13}&\; &\epsilon_{14}& \cr
&\epsilon_{12}&\; &0& \;  &\epsilon_{23}& \;&\epsilon_{24}& \cr
&\epsilon_{13}&\;  &\epsilon_{23}& \; &0&  \; &\epsilon_{34}& \cr
&\epsilon_{14}& \;&\epsilon_{24}&  \; &\epsilon_{34}& \; &0& 
\end{pmatrix}.
$$
For  each $i\ge 2$ we multiply the $i$th row of $A_a$ by  $\epsilon_{1i}$ and use  relation \ref{signeq}. Thus we obtain the matrix
$$B_a :=\begin{pmatrix}
&0 & \; &\epsilon_{12}& \; &\epsilon_{13}&\; &\epsilon_{14}& \cr
&1&\; &0& \;  &\xi \epsilon_{13}& \; &\xi \epsilon_{14}& \cr
&1& \;&\xi \epsilon_{12}& \; &0&  \;&\xi \epsilon_{14}& \cr
&1& \; &\xi \epsilon_{12}&  \;&\xi \epsilon_{13}& \; &0&
\end{pmatrix}.
$$

For  each $j\ge 2$ we multiply the $j$th column  of $B_a$ by  $\epsilon_{1j}$ and use  the fact that  $\epsilon_{1i}^2=1$. So  we obtain the matrix
$$C_a :=\begin{pmatrix}
&0 & \; &1& \; &1&\; &1& \cr
&1&\; &0& \;  &\xi & \; &\xi & \cr
&1& \;&\xi & \; &0&  \;&\xi & \cr
&1& \; &\xi&  \;&\xi & \; &0&
\end{pmatrix}.
$$
Multiplying the first row and the first  column of $C_a$ by $\xi$, we can see that  
$$\det C_a =\xi^2 \det T_{0,1}= t_{0,1}(4) = -3.$$
 But on the other hand 
$\det C_a  =( \epsilon_{12} \epsilon_{13} \epsilon_{14})^2 \det A_a =\det A_a$.
Accordingly we may assume that $ A_a= T_{0,\xi}$. The same argument applies for any
$d>4$.
Hence  the existence in Proposition  \ref{DetNuijToeplProp}  follows. 
To  proof the uniqueness, 
 note that  $\alpha$ and $\beta^2$ are uniquely determined. Clearly the equation $ a_3= \frac{1}{3!} (\alpha-\beta)^2(\alpha+2\beta)$ uniquely determines  $\beta$.
\end{proof}

As a consequence  we obtain the following characterization of   Nuij sequences
which admit  universal   determinantal representations.

\begin{mainB}\label{Nuijseqthm)}
A sequence 
 $ a= (a_1, \dots, a_d) \in \R^d$ is  a 
 Nuij sequence with a  universal   determinantal representation if and only if 
 there exits $\alpha,\beta \in \R$ such that
 \begin{equation}\label{Nuijseqeq}
a_i = \frac{1}{i!}t_{\alpha,\beta}(i), \, i=1, \dots, d.
\end{equation}

\end{mainB}
\begin{proof}  If $T_{\alpha,\beta}$ is a special Toeplitz matrix,  then for any  hyperbolic polynomial
$p(z)= (z+\lambda_1)\dots (z+\lambda_d)$ we have 
a pencil of   polynomials
$$
p_a(z,s):  = p +\sum_{k=1}^d a_ks^kp^{(k)}(z)= \det (zI+D+sT_{\alpha,\beta})
$$
where  $a_i = \frac{1}{i!}t_{\alpha,\beta}(i),$ and $D$ is a diagonal matrix with entries
$\lambda_1,\dots ,\lambda_d$.  So the  sequence 
 $ a= (a_1, \dots, a_d)$ is  a 
 Nuij sequence with a  universal   determinantal representation. Conversely, if 
 $ a= (a_1, \dots, a_d) \in \R^d$ is  a 
 Nuij sequence with a  universal   determinantal representation, then by Proposition \ref{DetNuijToeplProp} 
 the associated matrix can be chosen as a  special Toeplitz matrix $T_{\alpha,\beta}$. Hence 
 $a_i = \frac{1}{i!}t_{\alpha,\beta}(i)$.
\end{proof}

\begin{exam}
Note that the original   Nuij sequence  $a=(1,0, \dots , 0)$ has a universal   determinantal representation.
Indeed,  $T_{1,1}$, which has all entries equal to $1$, is  the matrix  associated to  this sequence. Note, that this also proves Nuij's Theorem \ref{NuijTheorem}.

\end{exam}

\begin{rema}
A composition of  the original Nuij sequence  $a=(1,0, \dots , 0)$ with itself gives  a Nuij sequence
$b=(2,1,0, \dots , 0)$  which has no universal determinantal  representation for $d\ge 3$.
Indeed,  if  there  exist  $\alpha,\beta \in \R$ such that $b_i = \frac{1}{i!}t_{\alpha,\beta}(i)$, $i=1,2,3$, then
$\alpha =2$ and $\alpha^2 -\beta^2 = 2$. Hence $\beta =\pm\sqrt{2}$. But, then
$6b_3 =\alpha^3+ 2\beta^3 - 3\alpha\beta^2 \ne 0$,  so $b_3\ne 0,$ which is a contradiction.
\end{rema}

\end{document}